\documentclass[12pt]{amsart}

\usepackage{amssymb,amsmath,amsthm, url}
\usepackage{enumitem}
\usepackage[all]{xy}
\usepackage[T1]{fontenc}
\usepackage[french,english]{babel}

\numberwithin{subsection}{section}

\allowdisplaybreaks[1]

\newtheorem*{namedtheorem}{\theoremname} \newcommand{\theoremname}{testing}

\newtheorem{theorem}{Theorem}
\newtheorem{proposition}[theorem]{Proposition}
\newtheorem{proposition-definition}[theorem] {Proposition-Definition}
\newtheorem{corollary}[theorem]{Corollary} 
\newtheorem{lemma}[theorem]{Lemma}

\theoremstyle{definition}

\newcommand\Gal{\operatorname{Gal}}

\newcommand\Char{\operatorname{char}}

\DeclareMathOperator{\Aut}{Aut}
\newcommand{\GL}{\mathrm{GL}}

\newcommand\FF{\mathbb{F}}

 \newcommand\PP{\mathbb{P}}

\newcommand\ZZ{\mathbb{Z}}   
\newcommand\bbZ{\mathbb{Z}}  \newcommand\bbF{\mathbb{F}}

\newcommand\arr{\ifinner\to\else\longrightarrow\fi}

\newcommand\arrto{\ifinner\mapsto\else\longmapsto\fi}

\def\displaytimes_#1{\mathrel{\mathop{\times}\limits_{#1}}}

\def\displayotimes_#1{\mathrel{\mathop{\bigotimes}\limits_{#1}}}

\newcommand\Spec{\operatorname{Spec}}

\newlength{\ignora}

\newcommand{\trdeg}{\operatorname{trdeg}}

\newcommand{\rdim}{\operatorname{rdim}}
\newcommand{\ed}{\operatorname{ed}}

\newcommand{\bbP}{\mathbb{P}}

\begin{document}
\title[Essential dimension in prime characteristic]
{Essential dimension of finite groups\\in prime characteristic}

\author[Reichstein]{Zinovy Reichstein$^\dagger$}

\author[Vistoli]{Angelo~Vistoli$^\ddagger$}

\address[Reichstein]{Department of Mathematics\\ University of British
Columbia \\ Vancouver, B.C., Canada V6T 1Z2}
\email{reichst@math.ubc.ca}

\address[Vistoli]{Scuola Normale Superiore\\ Piazza dei Cavalieri 7\\ 56126
Pisa\\ Italy}
\email{angelo.vistoli@sns.it}

\begin{abstract} 
Let $F$ be a field of characteristic $p > 0$ and $G$ be a smooth finite algebraic group over $F$.
We compute the essential dimension $\ed_F(G; p)$ of $G$ at $p$. That is, we show that 
\[ \ed_{F}(G; p) = \begin{cases} \text{$1$, if $p$ divides $|G|$, and} \\
\text{$0$, otherwise.} \end{cases} \]
\end{abstract}

\subjclass[2010]{20G15, 14G17}

\keywords{Essential dimension, versal torsor, Nottingham group, reduction of structure, Serre's Conjecture I}

\thanks{$^\dagger$Partially supported by
National Sciences and Engineering Research Council of
Canada Discovery grant 253424-2017.}
\thanks{$^\ddagger$Partially supported by research funds from 
the Scuola Normale Superiore.}

\thanks{
The authors are grateful to the Collaborative Research Group in Geometric and Cohomological Methods in Algebra at the Pacific Institute for the Mathematical Sciences for their support of this project.}

\maketitle

 \vspace{-1cm}

\begin{otherlanguage}{french}
\begin{abstract} 
Soit F un corps de caract\'eristique $p > 0$, et soit $G$ un groupe 
alg\'ebrique fini \'etale sur $F$.  On calcule la dimension essentielle 
de $G$ en $p$, que l'on note $\ed_F(G;p)$.  Plus pr\'ecis\'ement, on 
d\'emontre que
\[ \ed_{F}(G; p) = \begin{cases} \text{$1$, si $p$ divise $|G|$,} \\
\text{$0$, sinon.} \end{cases} \]
\end{abstract}
\end{otherlanguage}

\section{Introduction}

Let $F$ be a field and $G$ be an algebraic group over $F$. We begin by recalling the definition of the essential
dimension of $G$.

Let $K$ be a field containing $F$ and $\tau \colon T \to \Spec(K)$
be a $G$-torsor. We will say that $\tau$ descends to an intermediate subfield $F \subset K_0 \subset K$ if
$\tau$ is the pull-back of some $G$-torsor $\tau_0 \colon T_0 \to \Spec(K_0)$, i.e., if there exists a Cartesian diagram 
of the form
\[ \xymatrix{ T \ar@{->}[r] \ar@{->}[d]^{\tau} & T_0 \ar@{->}[d]^{\tau_0} &  \\
\Spec(K) \ar@{->}[r]  & \Spec(K_0) \ar@{->}[r] & \Spec(F). } \]
The essential dimension of $\tau$, denoted by $\ed_F(\tau)$, is the smallest value of the transcendence degree $\trdeg(K_0/F)$
such that $\tau$ descends to $K_0$. The essential dimension of $G$, denoted by $\ed_F(G)$, is the maximal value of
$\ed_F(\tau)$, as $K$ ranges over all fields containing $F$ and $\tau$ ranges over all $G$-torsors $T \to \Spec(K)$.

Now let $p$ be a prime integer.  A field $K$ is 
called $p$-closed if the degree of every finite extension $L/K$ is 
a power of $p$. Equivalently, $\Gal(K^s/K)$  is a pro-$p$-group, where $K^s$ 
is a separable closure of $K$. For example, the field of real numbers 
is $2$-closed.
The essential dimension $\ed_F(G; p)$ of $G$ at $p$ is the maximal value of 
$\ed_F(\tau)$, where $K$ ranges over $p$-closed fields $K$ containing $F$, and
$\tau$ ranges over the $G$-torsors $T \to \Spec(K)$.
For an overview of the theory of essential dimension, 
we refer the reader to the surveys~\cite{icm} and \cite{merkurjev13}. 

The case where $G$ is a finite group
(viewed as a constant group over $F$) is of particular interest. 
A theorem of N.~A.~Karpenko and A.~S.~Merkurjev~\cite{km} asserts that in this case
\begin{equation} \label{e.km}
\ed_{F}(G; p) = \ed_{F}(G_{p}; p) = \ed_{F}(G_p) = \rdim_{F}(G_p)  \, ,
\end{equation}
provided that $F$ contains a primitive $p$-th root of unity $\zeta_p$.
Here $G_p$ is any Sylow $p$-subgroup of $G$, and
$\rdim_{F}(G_p)$ denotes 
the minimal dimension of a faithful representation of $G_p$ defined over $F$.
For example, assuming that $\zeta_p \in F$, 
$\ed_F (G) = \ed(G; p) = r$ if $G = (\bbZ/ p \bbZ)^r$, and
$\ed(G) = \ed(G; p) = p$ if $G$ is a non-abelian
group of order $p^3$.
Further examples can be found in~\cite{mr10}.

Little is known about essential 
dimension of finite groups over a field $F$ of characteristic $p > 0$.
A.~Ledet~\cite{ledet-p} 
conjectured that 
\begin{equation} \label{e.ledet-conjecture}
\ed_F(\bbZ/p^r \bbZ) = r
\end{equation}
for every $r \geqslant 1$. This conjecture 
remains open for every $r \geqslant 3$. 
In this paper we will prove the following surprising result.  

\begin{theorem} \label{thm.main}
Let $F$ be a field of characteristic $p > 0$ and $G$ be a smooth finite algebraic group over $F$. Then 
\[ \ed_{F}(G; p) = \begin{cases} \text{$1$, if $p$ divides $|G|$, and} \\
\text{$0$, otherwise.} \end{cases}
\]
\end{theorem}

In particular, 
Ledet's conjecture~\eqref{e.ledet-conjecture} fails dramatically
if essential dimension is replaced by essential dimension at $p$.
On the other hand, Theorem~\ref{thm.main} fails if $\ed(G; p)$ is replaced by $\ed(G)$; see~\cite{ledet1}.

Before proceeding with the proof of Theorem~\ref{thm.main}, we remark that the condition that $G$ is smooth cannot be dropped. 
Indeed, it is well known that $\ed_F(\mu_p^r; p) = r$ for any $r \geqslant 0$.
More generally, if $G$ is a group scheme of finite type
over a field $F$ of characteristic $p$ (not necessarily finite or smooth), 
then $\ed_F(G; p) \geqslant \dim (\mathcal{G}) - \dim (G)$, where $\mathcal{G}$ 
is the Lie algebra of $G$; see~\cite[Theorem 1.2]{tv}.

\smallskip
{\bf Acknowledgments:}  We are grateful to the referee for a thorough 
reading of the paper and numerous constructive suggestions, including 
an alternate proof of Lemma~\ref{lem.curve}. We would also like to thank 
D. Tossici and J.-P.~Serre for helpful comments.

\section{Versality}
\label{sect.versality}

Let $G$ be an algebraic group and $X$ be an irreducible $G$-variety (i.e., a variety with a $G$-action) over $F$.
We will say that the $G$-variety $X$ is \emph{generically free} if there exists a dense open subvariety $U$ of $X$ such that
the scheme-theoretic stabilizer $G_u$ of every geometric point $u$ of $X$ is trivial. Equivalently, there exists a $G$-invariant dense open subvariety $U'$
of $X$, which is the total space of a $G$-torsor; see~\cite[Section 5]{gms}.

Following~\cite[Section 5]{gms} and~\cite[Section 1]{dr}, we will say that $X$ is
{\em weakly versal} (respectively, {\em weakly $p$-versal}), if for every infinite field (respectively, every $p$-closed field) $E$,
and every $G$-torsor $T \to \Spec(E)$ there is a $G$-equivariant $F$-morphism $T \to X$. We will say that $X$
{\em versal} (respectively, $p$-versal), if every $G$-invariant dense open subvariety of $X$ is weakly versal (respectively, weakly $p$-versal). 

It readily follows from these definitions that $\ed(G)$ (respectively, $\ed(G; p)$) is the minimal dimension $\dim(X) - \dim(G)$, 
where the minimum is taken over all versal (respectively $p$-versal) generically free $G$-varieties $X$; see~\cite[Section 5.7]{gms}, 
\cite[Remark 2.6 and Section 8]{dr}. Our proof of Theorem~\ref{thm.main} will be based on the the following facts. 

\smallskip
(i) (\cite[Proposition 2.2]{dr}) Every $G$-variety $X$ with a $G$-fixed $F$-point is weakly versal.

\smallskip
(ii) (\cite[Theorem 8.3]{dr}) Let $X$ be a smooth geometrically irreducible $G$-variety. Then $X$ is weakly $p$-versal if and only if $X$ is $p$-versal.

\smallskip
\noindent
Combining (i) and (ii), we obtain:

\begin{proposition} {\rm (\cite[Corollary 8.6(b)]{dr})} 
\label{prop8.6}
Let $G$ be a finite smooth algebraic group over $F$.
If there exists a faithful geometrically irreducible $G$-variety $X$ with a smooth $G$-fixed $F$-point, then
$\ed(G; p) \leqslant \dim(X)$.
\end{proposition}

If we replace ``$p$-versal" by ``versal", then (ii) fails: a weakly versal $G$-variety does not need to be versal.
This is the underlying reason why both Proposition~\ref{prop8.6} and Theorem~\ref{thm.main} fail if $\ed(G; p)$ is replaced by $\ed(G)$.

\section{Proof of Theorem~\ref{thm.main}}
\label{sect.proof}

In this section we will prove Theorem~\ref{thm.main}, assuming Lemmas~\ref{lem.curve} and~\ref{lem.twisted} below.
We will defer the proofs of these lemmas to sections~\ref{sect.lem.curve} and~\ref{sect.twisted}, respectively.

By \cite[Lemma 4.1]{mr10}, if 
$G' \subset G$ is a subgroup of index prime to $p$, then
\begin{equation} \label{e.G'}
\ed_F(G; p) = \ed_F(G'; p).
\end{equation} 
In particular, if $p$ does not divide $|G|$, then taking $G' = \{ 1 \}$,
we conclude that  $\ed_{F}(G; p) = 0$.
On the other hand, if $p$ divides $|G|$, then 
$\ed_{F}(G; p) \geqslant 1$; see \cite[Proposition 4.4]{merkurjev09} or~\cite[Lemma 10.1]{lmmr}.
Our goal is thus to show that $\ed_F(G; p) \leqslant 1$. 

First let us consider the case where $G$ is a finite group,
viewed as a constant algebraic group over $F$. After replacing $G$ by a Sylow $p$-subgroup, 
we may assume that $G$ is a $p$-group. Moreover, since $\bbF_p \subset F$, 
$\ed_F(G; p) \leqslant \ed_{\bbF_p}(G; p)$. Thus, for the purpose of proving the inequality
$\ed_F(G; p) \leqslant 1$, we may assume that $F = \bbF_p$. In view of Proposition~\ref{prop8.6} it suffices to prove the following.

\begin{lemma} \label{lem.curve} For every finite constant $p$-group $G$ there exists a faithful $G$-curve defined over $\bbF_p$
with a smooth $G$-fixed $\FF_p$-point.
\end{lemma}

Now consider the general case, where $G$ is a smooth finite algebraic group over $F$. In other words, $G = \, ^\tau \Gamma$, where $\Gamma$ is a constant finite group,
$A = \Aut_{\rm grp}(\Gamma)$ is the group of automorphisms of $\Gamma$ and 
$\tau$ is a cocycle representing a class in $H^1(F, A)$. 

\begin{lemma} \label{lem.twisted}
{\rm (a)} $\ed_F(G) \leqslant \ed_F(\Gamma \rtimes A)$, $\quad$ 
{\rm (b)} $\ed_F(G; p) \leqslant \ed(\Gamma \rtimes A; p)$.
\end{lemma}

The semidirect product $\Gamma \rtimes A$ is a constant finite group.
Hence, as we showed above, $\ed_F(\Gamma \rtimes A; p) \leqslant 1$. Theorem~\ref{thm.main} now follows from Lemma~\ref{lem.twisted}(b).

\section{Proof of Lemma~\ref{lem.curve}}
\label{sect.lem.curve}

We will give two proofs: our original proof, extracted from the literature, and a self-contained proof suggested to us by the referee.

\begin{proof}
Recall that the Nottingham group $\Aut_0(\FF_p[[t]])$ is the group of automorphisms $\sigma$ of the algebra $\FF_p[[t]]$ of formal power series
such that $\sigma(t) = t + a_2 t^2 + a_2 t^3 + \ldots$, for some $a_2, a_3, \ldots \in \FF_p$.
By a theorem of of Leedham-Green and Weiss~\cite[Theorem 3]{camina}, every finite $p$-group $G$ embeds into $\Aut_0(\FF_p[[t]])$.
Fix an embedding $\phi \colon G \hookrightarrow \Aut_0(\FF_p[[t]])$.
By~\cite[Theorem 1.4.1]{katz}, there exists a smooth $G$-curve $X$ over $\FF_p$,
with an $\FF_p$-point $x \in X$ fixed by $G$, such that the $G$-action in the formal neighborhood of $x$ is given by $\phi$; 
see also~\cite[Section 2]{harbater} and~\cite[Theorem 4.8]{bcps}. Since $\phi$ is injective, the $G$-action 
on $X$ is faithful. 
\end{proof}

\begin{proof}[Alternative proof]
First consider the case, where $G = (\ZZ/ p \ZZ)^n$ is an elementary abelian $p$-group. Here we can construct $X$ 
as the cover of $\bbP^1$ (with function field $\FF_p(s)$) given by the compositum of $n$ linearly disjoint
Artin-Schreier extensions $\FF_p(s, t_i)/\FF_p(s)$ given by $t_i^p - t_i = f_i(s)$ (e.g., taking $f_i(s) = s^{pi + 1}$). 

Now consider a general finite $p$-group $G$. Denote the Frattini subgroup of $G$ by $\Phi$ and the quotient $G/\Phi$ by $(\ZZ/ p\ZZ)^n$.
Let $Y$ be the smooth curve and $Y \to \bbP^1$ be a $G/\Phi = (\ZZ/ p \ZZ)^n$-cover constructed in the previous paragraph, totally
ramified at a point $y \in Y(\FF_p)$ above $\infty \in \PP^1$. 
Let $E/\FF_p(s)$ be the $(\ZZ/ p \ZZ)^n$-Galois extension associated to this cover. By~\cite[Proposition II.2.2.3]{serre-gc},
the cohomological dimension of
$\FF_p(s)$ at $p$ is $\leqslant 1$. Consequently by~\cite[Propositions I.3.4.16]{serre-gc},
$E/\FF_p(s)$ lifts to a $G$-Galois extension $K/\FF_p(s)$ such that $K^{\Phi} = E$. Let $X$ be the smooth curve associated to $K$ and 
$x \in X(\overline{\FF_p})$ is a point above $y$:
\[ \xymatrix{ x  \ar@{->}[d]  &\in  & X \ar@{->}[d]   \\
              y  \ar@{->}[d]  & \in   & Y \ar@{->}[d]  \\
              \infty & \in &   \bbP_{\FF_p}^1.}
\]
We claim that $x$ is fixed by $G$; in particular, this will imply that $x \in X(\FF_p)$. Let $H$ be the stabilizer of $x$ in $G$. Since
$\Phi$ acts transitively on the fiber above $y$ in $X$, we have $\Phi \cdot H = G$. By Frattini's theorem 
(see, e.g.,~\cite[Theorem 5.2.12]{robinson}), $\Phi$ is the set of non-generators of $G$. We conclude that $H = G$,
as claimed. 
\end{proof}

\section{Proof of Lemma~\ref{lem.twisted}}
\label{sect.twisted}

We will make use of the following description of $\ed_F(G)$ and $\ed_F(G; p)$ in the case where $G$ is a finite 
algebraic group over $F$. Let $G \to \GL(V)$ be a faithful representation. A compression (respectively, a $p$-compression) of $V$ is a dominant $G$-equivariant rational map $V \dasharrow X$ (respectively, a dominant $G$-equivariant correspondence $V \rightsquigarrow X$ of degree prime to $p$),
where $G$ acts faithfully on $X$. Here by a correspondence we mean a $G$-equivariant subvariety $V'$ of $V \times X$ such that the $G$ transitively permutes the
irreducible components of $V'$, and the dimension of each component equals the dimension of $V$. The degree of this correspondence is defined as the degree of the projection $V' \to V$ to the first factor.

Recall that $\ed_F(G)$ (respectively, $\ed_F(G; p)$) equals the minimal value of $\dim(X)$ taken 
over all compressions $V \dasharrow X$ (respectively all $p$-compressions $V \rightsquigarrow X$). In particular,
these numbers depend only on $G$ and $F$ and not on the choice of the generically free 
representation $V$. For details, see~\cite{icm}. 

We are now ready to proceed with the proof of Lemma~\ref{lem.twisted}. To prove part (a),
let $V$ be a generically free representation of $\Gamma \rtimes A$ and let $f \colon V \dasharrow X$ be 
a $\Gamma \rtimes A$-compression, with
$X$ of minimal possible dimension. That is, $\dim_F(X) = \ed_F(\Gamma \rtimes A$). 
Twisting by $\tau$, we obtain a $G = \, ^{\tau} \Gamma$-equivariant map $^\tau f \colon ^{\tau} V \dasharrow \, ^{\tau} X$; see e.g., \cite[Proposition 2.6(a)]{fr}. 
Now observe that by Hilbert's Theorem 90, $^\tau V$ is a vector space with a linear action 
of $G = \, ^\tau \Gamma$ and  $^\tau f \colon \, ^{\tau} 
V \dasharrow \, ^{\tau} X$ is a compression. (To see that
the $G$-action on $^\tau V$ and $^\tau X$ are faithful, we may
pass to the algebraic closure $\overline{F}$ of $F$. 
Over $\overline{F}$, $\tau$ is split, so that $G = \Gamma$, 
$^\tau V = V$, $^\tau X = X$ 
and $^\tau f = f$, and it becomes obvious that the $G$-actions 
on $^{\tau} V$ and $^{\tau} X$ are faithful.) 
We conclude that $\ed_F(G) \leqslant \dim_F(\, ^\tau X) = \dim_F(X) = 
\ed_F(\Gamma \rtimes A)$, as desired.

The proof of part (b) proceeds along the same lines. The starting point is a $p$-compression 
$f\colon  V \rightsquigarrow X$ with $X$ of minimal possible dimension, $\dim_F(X) = \ed_F(\Gamma \rtimes A; p)$. 
We twist $f$ by $\tau$ to obtain a $p$-compression $^\tau f \colon \, ^{\tau} 
V \rightsquigarrow \, ^{\tau} X$ of the linear action of $G = \, ^\tau \Gamma$ on $^{\tau} V$.
The rest of the argument is the same as in part (a). 
This completes the proof of Lemma~\ref{lem.twisted} and thus of Theorem~\ref{thm.main}.
\qed

\section{An application}
\label{sect.application}

In this section $G$ will denote a connected reductive linear algebraic group over a field $F$. It is shown 
in~\cite[Theorem 1.1(c)]{cgr1} that there exists a finite $F$-subgroup $S \subset G$ such that 
every $G$-torsor over every field $K/F$ admits reduction of structure to $S$; see also \cite[Corollary 1.4]{cgr2}.
In other words, the map $H^1(K, S) \to H^1(K, G)$ is surjective for every field $K$ containing $F$. If this happens, we will
say that ``$G$ admits reduction of structure to $S$\,''.

We will now use Theorem~\ref{thm.main} to show that if $\Char(F) = p > 0$ and $p$ is a torsion prime for $G$, 
then $S$ cannot be smooth. For the definition of torsion primes,
a discussion of their properties and further references, 
see~\cite{serre2}. Note that
by a theorem of A.~Grothendieck~\cite{grothendieck},
if $G$ is not special (i.e., if $H^1(K, G) \neq \{ 1 \}$ for 
some field $K$ containing $F$), then $G$ has at least one torsion prime; 
see also~\cite[1.5.1]{serre2}. 

\begin{corollary} \label{cor.main}
Let $G$ be a connected reductive linear algebraic group over an algebraically closed 
field $F$ of characteristic $p > 0$. 

(a) If $S$ is a smooth finite subgroup of $G$ defined over $F$, then 
the natural map
   \[
   f_{K}\colon H^1(K, S) \to H^1(K, G)
   \]
is trivial for any $p$-closed field $K$ containing $F$. In other words,
$f_K$ sends every $\alpha \in H^1(K, S)$ to $1 \in H^1(K, G)$.

(b) If $p$ is a torsion prime for $G$, then $G$ does not admit 
reduction of structure to any smooth finite subgroup.
\end{corollary}

\begin{proof} (a) Let $\alpha \in H^1(K, S)$ and 
$\beta = f_K(\alpha) \in H^1(K, G)$. By Theorem~\ref{thm.main}, 
$\alpha$ descends to $\alpha_0 \in H^1(K_{0}, S)$ for some intermediate field
$F \subset K_0 \subset K$, where $\trdeg(K_0/F) \leqslant 1$.
Since $F$ is algebraically closed, $\dim(K_0)\leqslant 1$;
see~\cite[Sections II.3.1-3]{serre-gc}. By Serre's Conjecture I
(proved by R.~Steinberg~\cite{steinberg} for a perfect field $K_0$ 
and by A.~Borel and T.~A.~Springer~\cite[\S 8.6]{bs} for an arbitrary
$K_0$ of dimension $\leqslant 1$), 
$H^1(K_0, G) = \{ 1 \}$.
Tracing through the diagram
 \[ \xymatrix{    
H^1(K_0, S) \ar@{->}[rrr]^{f_{K_0}} \ar@{->}[ddd] &   &   & H^1(K_0, G) = \{ 1 \} 
\ar@{->}[ddd] \\
                 & \alpha_0 \ar@{|->}[d] \ar@{|->}[r]  & 1 
\ar@{|->}[d] &  \\
                 & \alpha \ar@{|->}[r] & \beta     &               \\
H^1(K, S) \ar@{->}[rrr]^{f_K}  &   &   &  H^1(K, G), }  \]
we see that $\beta = 1$, as desired.

\smallskip
(b) If $p$ is a torsion prime for $G$, then $H^1(K, G) \neq \{ 1 \}$ for some $p$-closed field $K$ containing $F$; see~\cite[Proposition 4.4]{merkurjev09}.
In view of part (a), this implies that $f_K$ is not surjective.
\end{proof}

%


\begin{thebibliography}{LMMR13}

\bibitem[BCPS17]{bcps}
F.~M.~Bleher, T. Chinburg, B.~Poonen, P.~Symonds,
{\em Automorphisms of Harbater-Katz-Gabber curves}, Math. Ann. {\bf 368} (2017), 
no.~1-2, 811--836. MR3651589

\bibitem[BS68]{bs}
A.~Borel\ and\ T.~A.~Springer, {\em Rationality properties of linear algebraic groups. II}, T\^ohoku Math. J. (2) {\bf 20} (1968), 443--497. MR0244259

\bibitem[CGR06]{cgr1} 
V. Chernousov, P. Gille\ and\ Z. Reichstein, {\em Resolving $G$-torsors by abelian base extensions}, J. Algebra {\bf 296} (2006), no.~2, 561--581. MR2201056

\bibitem[CGR08]{cgr2}
V. Chernousov, P. Gille\ and\ Z. Reichstein, {\em Reduction of structure for torsors over semilocal rings}, Manuscripta Math. {\bf 126} (2008), no.~4, 465--480. MR2425436

\bibitem[C97]{camina}
R.~Camina, {\em Subgroups of the Nottingham group}, J.~Algebra {\bf 196} (1997), no.~1, 101--113. MR1474165

\bibitem[CGR06]{cgr}
V.~Chernousov, P.~Gille\ and\ Z.~Reichstein, {\em Resolving $G$-torsors by abelian 
base extensions}, J.~Algebra {\bf 296} (2006), no.~2, 561--581. MR2201056

\bibitem[DR15]{dr}
A.~Duncan\ and\ Z.~Reichstein, {\em Versality of algebraic group actions and rational points on twisted varieties}, J. Algebraic Geom. {\bf 24} (2015), no.~3, 499--530. MR3344763

\bibitem[FR17]{fr}
M.~Florence, Z.~Reichstein, {\em The rationality problem for forms of moduli spaces of stable marked curves of positive genus}, arXiv:1709.05696.

\bibitem[Gr58]{grothendieck}
A.~Grothendieck, {\em Torsion homologique et sections rationnelles}, in: 
\emph{Anneaux de Chow et Applications}
S\'eminaire Claude Chevalley, {\bf 3} (1958), expos\'e 5, 1--29.

\bibitem[Ha80]{harbater}
D.~Harbater, {\em Moduli of $p$-covers of curves}, Comm. Algebra {\bf 8} (1980), no.~12, 1095--1122. MR0579791

\bibitem[KM08]{km}
N.~A.~Karpenko\ and\ A.~S.~Merkurjev, {\em Essential dimension of finite $p$-groups}, Invent. Math. {\bf 172} (2008), no.~3, 491--508. MR2393078

\bibitem[Ka86]{katz}
N.~M.~Katz, {\em Local-to-global extensions of representations of fundamental groups}, Ann. Inst. Fourier (Grenoble) {\bf 36} (1986), no.~4, 69--106. MR0867916

\bibitem[Le04]{ledet-p}
A.~Ledet, On the essential dimension of $p$-groups, in {\it Galois theory and modular forms}, 159--172, Dev. Math., 11, Kluwer Acad. Publ., Boston, MA, 2004. MR2059762

\bibitem[Le07]{ledet1}
A. Ledet, {\em Finite groups of essential dimension one}, J. Algebra {\bf 311} (2007), no.~1, 31--37. MR2309876

\bibitem[LMMR13]{lmmr}
R.~L\"otscher, M.~MacDonald, A.~Meyer, and Z.~Reichstein,
{\em Essential $p$-dimension of algebraic groups whose connected component 
is a torus}, Algebra Number Theory {\bf 7} (2013), no.~8, 1817--1840. MR3134035

\bibitem[Me09]{merkurjev09}
A.~S.~Merkurjev, Essential dimension, in {\it Quadratic forms---algebra, arithmetic, and geometry}, 299--325, Contemp. Math., 493, Amer. Math. Soc., Providence, RI, 2009. MR2537108

\bibitem[Me13]{merkurjev13}
A.~S.~Merkurjev, {\em Essential dimension: a survey}, 
 Transform. Groups, 18 (2013), no.~2, 415--481, 2013.

 \bibitem[MR09]{mr09} A.~Meyer and Z.~Reichstein,
 {\em The essential dimension of the normalizer of a maximal torus in the
  projective linear group}, Algebra Number Theory 3, no. 4 (2009), 467--487.

\bibitem[MR10]{mr10}
A. Meyer\ and\ Z. Reichstein, {\em Some consequences of the 
Karpenko-Merkurjev theorem}, Doc. Math. {\bf 2010}, Extra vol.: Andrei A. Suslin sixtieth birthday, 445--457. MR2804261

\bibitem[Rei10]{icm}
Z.~Reichstein, Essential dimension, 
in {\it Proceedings of the International Congress of Mathematicians. 
Volume II}, 162--188, Hindustan Book Agency, New Delhi, 2010.

\bibitem[Ro96]{robinson}
D. J. S. Robinson, {\it A course in the theory of groups}, second edition, Graduate Texts in Mathematics, 80, Springer-Verlag, New York, 1996. MR1357169

\bibitem[Se97]{serre-gc}
J.-P.~Serre, {\it Galois cohomology}, translated from the French by Patrick Ion and revised by the author, Springer-Verlag, Berlin, 1997. MR1466966

\bibitem[Se00]{serre2}
J.-P.~Serre, {\em Sous-groupes finis des groupes de Lie}, Ast\'erisque No. 266 (2000), Exp.\ No.\ 864, 5, 415--430. MR1772682

\bibitem[Se03]{gms}
J.-P.~Serre, {\em Cohomological invariants, {W}itt invariants, and trace forms}, 
in Cohomological invariants in {G}alois cohomology, {\bf 28}
   Univ. Lecture Ser., pp. 1--100. Amer. Math. Soc., Providence, RI, 2003. Notes by Skip Garibaldi.

\bibitem[St65]{steinberg}
R. Steinberg, {\em Regular elements of semisimple algebraic groups}, Inst. Hautes \'Etudes Sci. Publ. Math., no. 25 (1965), 49--80. MR0180554
 
\bibitem[TV13]{tv}
D. Tossici\ and\ A. Vistoli, {\em On the essential dimension of infinitesimal group schemes}, 
Amer. J. Math. {\bf 135} (2013), no.~1, 103--114. MR3022958
\end{thebibliography}
\end{document}